\documentclass{amsart}
\usepackage{amsmath, amssymb, amsthm, fullpage}
\theoremstyle{plain} 
\newtheorem{theorem}{Theorem}
\newtheorem{lemma}[theorem]{Lemma}
\numberwithin{equation}{section}

\def\RR{\mathbb R}

\begin{document}

\title{On the Waring-Goldbach Problem for tenth powers}

\author{Mayank Pandey}
\date{}

\maketitle

\section{Introduction}

Define $H(k)$ to be the least $s$ such that all sufficiently large $n$ satisfying some congruence condition are the sum of $s$ $k^{th}$ powers of primes.

Work done by Vinogradov \cite{Vin1}, Hua \cite{Hua1, Hua2}, and Davenport \cite{Dav}, 
     and later work by Thanigasalam \cite{Tha1, Tha2, Tha3} and Vaughan \cite{Vau1} resulted in the knowledge that 

    $$H(4) \le 15, H(5) \le 23, H(6) \le 33, H(7)\le 47, H(8) \le 63, H(9) \le 83, H(10)\le 107.$$
Recently, Kawada and Wooley \cite{kw} and Kumchev \cite{kum} showed that 
    
    $$H(4)\le 14, H(5)\le 21, H(7)\le 46,$$
and even more recently, Zhao \cite{Zha} has shown that 

    $$H(4) \le 13, H(6) \le 32.$$\\
The purpose of this paper is to establish similar results for tenth powers.\\ 
In particular, we shall obtain the following:

\begin{theorem}\footnote{Since the writing of this paper, the bound $H(10)\le 89$ has been achieved in \cite{kwo}}

$H(10)\le 105$. In particular, every sufficiently large integer congruent to 6 modulo 33 is the sum of 105 tenth powers of primes. 

\end{theorem}

The above result will be established with the Hardy-Littlewood method. The results which allow us to improve the previous bound of 
$107$ are the new Weyl sum estimates obtained by Kumchev \cite{kum2} and the mean value estimates in \cite{Tha4}.
\section{Preliminaries}

\subsection{Notation}

In all cases, unless otherwise specified, $p$ will refer to a prime. 

For $x\in\RR$, let $e(x) = e^{2\pi i x}.$

As usual, we will write $O(f)$ to denote some quantity bounded above by $C|f|$ for some $C$, and $f\ll g$ means $|f| \le Cg$ for some $C$, and $\epsilon$ will refer to a sufficiently small positive real number. All sums will be over the natural numbers unless otherwise specified.

We write $$\mathfrak{M}(q, a; Q, P) = \left\{\alpha\in[0, 1] : |q\alpha - a|\le \frac{Q}{P}\right\}.$$
We define the primorial $Q\#$ of $Q > 0$ to be the product of all primes less than or equal to $Q$. 
Also, $c$ will refer to some constant, and will not necessarily be the same each time it is mentioned in the paper.

Call a set $\{\lambda_1, \dots, \lambda_k\}\subset\RR^+$ admissible if the number of solutions $S$ to 
$$\sum_{i\le k} (x_i^k - y_i^k) = 0$$ 
satisfying $P^{\lambda_i} < x_i, y_i\le 2P^{\lambda_i}$ satisfies $S\ll P^{\lambda_i + \dots + \lambda_k + \epsilon}$.
\subsection{Exponential Sum Estimates}  

\begin{lemma}

For all $\alpha$ for which for all coprime $0 \le a \le q \le P^{1/4}$ s.t.  $|q\alpha - a| > P^{1/4}P^{-10}$,

\begin{equation}
\sum_{P < p \le 2P} e(\alpha p^{10}) \ll P^{1 - 1/480 + \epsilon}
\end{equation}
\end{lemma}

\begin{proof}

We have that by Theorem 2 in Kumchev \cite{kum2}, 
$$\sum_{n \le P} \Lambda(n) e(\alpha n^{10}) = \sum_{p \le P} 
\log p e(\alpha p^{10}) + O(\sqrt{P})\ll P^{1 - 1/480 + \epsilon}.$$

Then, by partial summation, we have that 
$$\sum_{p\le P}e(\alpha p^{10}) = \frac{1}{\log P}\sum_{p\le P}\log p e(\alpha p^{10}) - 
\int_2^P \left(-\frac{1}{t\log ^2 t}\right)\sum_{p\le t} \log p e(\alpha p^{10}) dt.$$

Note that 
$$\int_2^P \left(\frac{1}{t\log ^2 t}\right)\sum_{p\le t} \log p e(\alpha p^{10}) dt \le $$
$$\int_2^P\left(\frac{1}{t\log^2 t}\right) \left\lvert \sum_{p\le t} \log p e(\alpha p^{10})\right\rvert dt 
\ll \int_2^P \frac{t^{\epsilon} dt}{t^{1/480}\log^2 t} \ll P^{1 - 1/480 + \epsilon}.$$

The desired result follows.

\end{proof}

\begin{lemma} 
For some $a, q,\alpha$ satisfying $(a, q)  =1$, $|q\alpha - a| \le QP^{-10}$ for some $q \le Q\le P$,

$$\sum_{P < p \le 2P} e(\alpha p^{10}) \ll 
q^{\epsilon} (\log P)^c\left(P\left(q + P^{10}|q\alpha - a|\right)^{-1/2} + P^{11/20}\left(q + P^5|q\alpha - a|\right)^{-1/2})\right).$$

\end{lemma}

\begin{proof}
This is just Lemma 5.6 in Kumchev \cite{kum} with $M = 1/2$, $z = \sqrt{2P}$.
\end{proof}

\section{Mean-Value Estimates}

\begin{lemma}

There exist admissible exponents $1 = \lambda_1, \dots, \lambda_{51}$ satisfying 
\begin{equation}
\alpha_{51} = \frac{\lambda_1 + \dots + \lambda_{51}}{10} > 0.999553 > 1 - \frac{1}{2230}.
\end{equation}

\end{lemma}

\begin{proof}

This follows from Lemma 18 in Thanigasalam \cite{Tha5} and (10.5) in Thanigasalam \cite{Tha3}.

\end{proof}

For $1\le j\le 51$, write $P_j = P^{\lambda_j}$ where the $\lambda_i$ are as in Lemma 4, and let
$$f_j(\alpha) = \sum_{P_j < n\le 2P_j} e(\alpha n^{10})$$

Then, since the $\lambda_j$ are admissible, we have that the following, which is the main result of this section holds:

\begin{lemma}
We have that 
\begin{equation}
\int_0^1 |f_1(\alpha)\dots f_{51}(\alpha)|^2 d\alpha\ll P^{10\alpha_{51} + \epsilon},
\end{equation}

where $\alpha_{51}$ is the constant mentioned in Lemma 4. 
\end{lemma}

\section{Proof of the main theorem}

Let $N$ be some large integer congruent to $6\pmod{33}$, let $B$ be a sufficiently large real number, and set 
$$P = \frac{1}{2}N^{1/10}, \hspace{1cm} X = P_1^{5}P_2^2\dots P_{51}^2 N^{-1}, \hspace{1cm} L = \log^B P.$$

Let $$\mathfrak{N}(q, a) = \mathfrak{M}(q, a;L, P^{10})\hspace{1cm}
\mathfrak{N} = \underset{(a, q = 1)}{\bigcup_{0\le a\le q \le \log^B P}}\mathfrak{N}(q, a)$$
    $$\mathfrak{M} = \underset{(a, q = 1)}{\bigcup_{0\le a\le q \le P^{1/4}}} \mathfrak{M}(q, a; P^{1/4}, P^{10})$$
and let $\mathfrak{m} = [0, 1)\setminus \mathfrak{M}$, $\mathfrak{n} = [0, 1)\setminus \mathfrak{N}$

For $1 \le i\le 51$, define $$g_i(\alpha) = \sum_{P < p \le 2P} e(\alpha p^{10}).$$ 

For some measurable $\mathfrak{B}\subseteq [0, 1)$, write 
$$R(N; \mathfrak{B}) = \int_0^1 g_1(\alpha)^5 g_2(\alpha)^2\dots g_{51}(\alpha)^2 d\alpha.$$

Let 
$$R(N) = |\{(p_1, \dots, p_{105}) : p_1^{10} + \dots + p_{105}^{10} = N\}|, $$
   for primes $p_1, \dots, p_{105}$ satisfying
   $$P_1 < p_1, p_2, p_3, p_4, p_5\le 2P_1, P_i < p_{2i + 2}, p_{2i + 3}\le 2P_i \text{ for } 2\le i\le 51.$$

Then, by orthogonality, we have that $R(N) =  R(N; [0, 1)).$

Note that in order to prove Theorem 1, it is sufficient to show that for all sufficiently large $N$, $R(N) > 0$. 

\subsection{The major arcs}

In this section, we shall consider the contribution to $R(N)$ from the major arcs $\mathfrak{N}$.

Write $$S(q, a) = \underset{(k, q) = 1}{\sum_{1\le k\le q}}e\left(\frac{ak^{10}}{q}\right), $$
$$v_i(\beta) = \int_{P_i}^{2P_i} \frac{e(\beta t^{10})}{\log t}dt, $$
	$$B(N, q) = \frac{1}{\phi(q)^{105}}\underset{(a, q) = 1}{\sum_{1\le a\le q}} S(q, a) e\left(\frac{-aN}{q}\right), $$
	$$\mathfrak{S}(N) = \sum_{q} B(n, q), $$
	$$J(N; \xi) = \int_{-\xi}^{\xi} v_1(\beta)^5v_2(\beta)^2\dots v_{51}(\beta)^2 e(-N\alpha)d\beta,$$
	and let $J(N) = J(N; \infty).$

	Note that by Theorem 12 in \cite{Hua2}, $\mathfrak{S}(N) \asymp 1$. We also have that $J(N) \asymp X\log^{-105} P.$

	We then have that by partial summation and the Siegel-Walfisz Theorem that for $\alpha\in\mathfrak{N}_0(q, a)$ 
	$$g_i(\alpha) = \phi(q)^{-1}S(q, a)v(\alpha - a/q) + O(P_i L^{-3})$$

	Therefore, since the measure of $\mathfrak{N}$ is $O(L^2 n^{-1})$  
	$$R(N; \mathfrak{N}) = \mathfrak{S}(N)J(N) + O(X L^{-1})\gg X \log^{-105} P$$

\subsection{The minor arcs}

In this section we shall bound the contribution from $\mathfrak{K} = \mathfrak{M}\cap \mathfrak{n}$ and $\mathfrak{m}$

\begin{lemma}
There exists $\eta > 0$ s.t. 
\begin{equation}R(N; \mathfrak{m})\ll XP^{-\eta + \epsilon}\end{equation}.
\end{lemma}

\begin{proof}
In fact, we shall prove that this is the case for all $\eta = 1/160 - 1/223$.

We have that by (2.1), $$\sup_{\alpha\in\mathfrak{m}} |g_1(\alpha)| \ll P^{1 - 1/480 + \epsilon}.$$

It then follows from (3.1) and (3.2) by considering the underlying diophantine equation 
that $$R(N; \mathfrak{m}) = \int_{\mathfrak{m}} g_1(\alpha)^5g_2(\alpha)^2\dots g_{51}(\alpha)^2 d\alpha$$
$$\ll \left(\sup_{\alpha\in\mathfrak{m}} |g_1(\alpha)|\right)^3\int_0^1 |f_1(\alpha)\dots f_{51}(\alpha)|^2d\alpha
\ll P^{3 - 1/160}(P_1\dots P_{51})^2 P^{10\alpha_{51} + \epsilon} $$ 
$$ \ll P^{1/223 - 1/160 + \epsilon} P^{-10}(P_1\dots P_{51})^2
\ll XP^{-\eta + \epsilon}$$ as desired.

\end{proof}

\begin{lemma}

We have that 
$$R(N; \mathfrak{K}) \ll XL^{-1} \log^c P.$$
\end{lemma}

\begin{proof}
Note that $\mathfrak{K}$ is the disjoint union of $\mathfrak{K}(q, a)$ for coprime $a, q$ satisfying $0 \le  a\le q \le P^{1/4}$, 
where $\mathfrak{K}(q, a) = \mathfrak{M}(q, a)\setminus \mathfrak{N}(q, a)$ for $q\le L$ and 
$\mathfrak{K}(q, a) = \mathfrak{M}(q, a)$ otherwise.

Then, it follows by applying Lemma 4 that 
$$\int_{\mathfrak{K}} g_1(\alpha)^5 g_2(\alpha)^2\dots g_{51}(\alpha)^2d\alpha$$
$$\ll \int_{\mathfrak{K}} |g_1(\alpha)|^5|g_2(\alpha|^2|g_3(\alpha)||g_3(\alpha)||g_4(\alpha)\dots g_{51}(\alpha)|^2d\alpha$$
$$\ll Xn\sum_{q\le P^{1/4}}\underset{(a, q) = 1}{\sum_{1\le a\le q}}\int_{\mathfrak{K}(q, a)}
\frac{\log^c P d\alpha}{q^4(1 + n|\alpha - a/q|)^2}.$$
The desired result follows. 
\end{proof}

Now, it follows from this and (4.1), by making $B$ sufficiently large, that 
$$R(N) = R(N; \mathfrak{N}) + R(N; \mathfrak{K}) + R(N; \mathfrak{m})\gg X\log^{-105}P, $$
so Theorem 1 holds. 

\section{Acknowledgements}
The author is thankful to T. D. Wooley and  D. Goldston for providing corrections and suggestions.

\end{document}